\documentclass{amsart}
\usepackage{amssymb,latexsym}
\usepackage{amsmath}
\usepackage{graphicx}
\usepackage{enumerate}
\usepackage{newlattice}
\usepackage{gensymb}
\usepackage{xspace}

\newcommand{\swing}{\mathbin{\raisebox{2.0pt}
       {\rotatebox{160}{$\curvearrowleft$}}}}
\newcommand{\exswing}{\stackrel{\textup{ ex}}{\swing}}
\newcommand{\inswing}{\stackrel{\textup{\,in}}{\swing}}
\newcommand{\grrel}{\mathbin{\gr}}

\theoremstyle{plain}
\newtheorem{theorem}{Theorem}
\newtheorem*{theoremn}{Theorem}
\newtheorem{lemma}[theorem]{Lemma}
\newtheorem{corollary}[theorem]{Corollary}
 
\theoremstyle{definition}
\newtheorem{definition}[theorem]{Definition}

\newcommand{\pr}[1]{\tup{(}#1\tup{)}}
\newcommand{\Col}[1]{\tup{col}(#1)}
\newcommand{\col}[1]{\tup{col}(#1)}

\newcommand{\traj}[1]{\tup{traj}{(#1})}

 \newcommand{\seven}{\SfS 7}
 
\begin{document}
\title[The Swing Lemma and $\E C_1$-diagrams]
{Using the Swing Lemma and $\E C_1$-diagrams 
for congruences of planar semimodular lattices}
\author[G.\ Gr\"atzer]{George Gr\"atzer}
\email{gratzer@me.com}
\urladdr{http://server.maths.umanitoba.ca/homepages/gratzer/}
\address{University of Manitoba}
\date{June 6, 2021}

\begin{abstract} 
A planar semimodular lattice $K$ is \emph{slim} 
if $\SM{3}$ is not a sublattice of~$K$.
In a recent paper, G. Cz\'edli found four new properties  
of congruence lattices of slim, planar, semimodular lattices,
including the \emph{No Child Property}: 
\emph{Let~$\mathcal{P}$ be the ordered set of join-irreducible congruences of $K$.
Let $x,y,z \in  \mathcal{P}$ and let $z$ be a~maximal element of $\mathcal{P}$.
If  $x \neq y$ and $x, y \prec z$ in $\mathcal{P}$, 
then there is no element $u$ of $\mathcal{P}$ such that $u \prec x, y$ in $\mathcal{P}$.}

We are applying my Swing Lemma, 2015,
and a type of standardized diagrams of Cz\'edli's, to verify his four properties.
\end{abstract}



\maketitle    

\section{Introduction}\label{S:Introduction}

Let $K$ be a planar semimodular lattice. 
We call the lattice $K$ \emph{slim} if $\SM{3}$ is not a~sublattice of~$K$.
In the paper \cite[Theorem 1.5]{gG14a}, I found a property of congruences of  slim, planar, semimodular lattices.
In the same paper (see also Problem 24.1 in G. Gr\"atzer~\cite{CFL2}),
I~proposed the following:

\vspace{6pt}

\tbf{Problem.} Characterize the congruence lattices of slim  planar semimodular lattices.

\vspace{6pt}

G. Cz\'edli ~\cite[Corollaries 3.4, 3.5, Theorem 4.3]{gCa} found four new properties  
of congruence lattices of slim, planar, semimodular lattices.

\begin{theoremn}\label{T:main}
Let $K$ be a slim, planar, semimodular  lattice 
with at least three elements and let~$\E P$ be
the ordered set of join-irreducible congruences of $K$.

\begin{enumeratei}
\item \emph{Partition Property:} 
The set of maximal elements of $\E P$ can be divided into the disjoint union 
of two nonempty subsets such that no two distinct elements 
in the same subset have a common lower cover.\label{E:LC} 
\item \emph{Maximal Cover Property:}  
If $v \in \E P$ is covered by a maximal element $u$ of $\E P$, 
then $u$ is not the only cover of $v$.
\item \emph{No Child Property:} 
Let $x \neq y \in \E P$ and let $u$ be a maximal element of $\E P$. 
If $x,y \prec u$ in $\E P$, 
then there is no element $z \in \E P$ such that $z \prec x, y$ in $\E P$.
\item \emph{Four-Crown Two-pendant Property:}
There is no cover-preserving embedding of the ordered set $\E R$ in Figure~\ref{F:notation}
into $\E P$ satisfying the property\tup{:} any maximal element of~$\E R$
is a maximal element of $\E P$.
\end{enumeratei}
\end{theoremn}

In this paper, we will provide a short and direct proof of this theorem 
using only the Swing Lemma and $\E C_1$-diagrams, see Section~\ref{S:Tools}.

\begin{figure}[t!]
\centerline
{\includegraphics[scale=1]{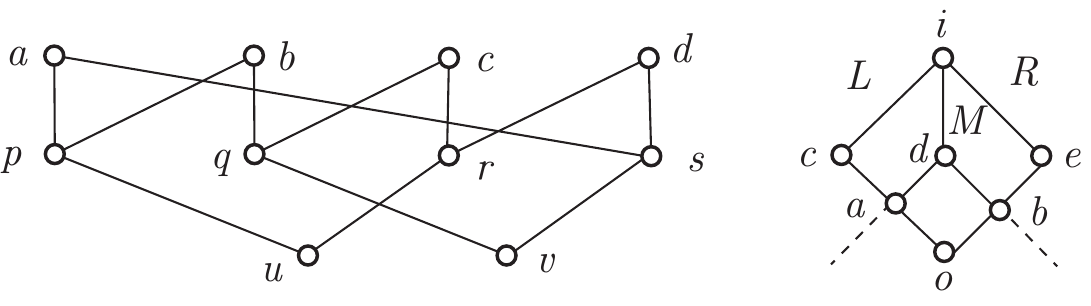}}
\caption{The Four-crown Two-pendant ordered set $\E R$ with notation; 
the covering $\SfS 7$ sublattice with edge and element notation}
\label{F:notation}
\end{figure}

\subsection*{Outline} 
Section~\ref{S:Motivation} provides the motivation for Cz\'edli's Theorem. 
Section~\ref{S:Tools} provides the tools we need: 
the Swing Lemma,  $\E C_1$-diagrams, and forks.
Section~\ref{S:partition} proves the Partition Property, 
Section~\ref{S:Maximal} does the Maximal Cover Property,
while Section~\ref{S:Child} verifies the No Child Property.
Finally,  The Four-Crown Two-pendant Property is proved in Section~\ref{S:Crown}.

\section{Motivation}\label{S:Motivation}
In my paper \cite{GLS98a} with H. Lakser and E.\,T. Schmidt, 
we proved that every finite distributive lattice $D$ can be represented
as the congruence lattice of a semimodular lattice $L$.
To our surprise, the semimodular lattice $K$ we constructed was \emph{planar}.

G.~Gr\"atzer and E.~Knapp~\cite{GKn07}--\cite{GK10} 
started the study of planar semimodular lattices. 
I continued it with my ``Notes on planar semimodular lattices'' series (started with Knapp):  
\cite{gG13}, \cite{GW10} (with T. Wares),  \cite{CG12} (with G. Cz\'edli), 
\cite{gG19}, \cite{gG21b}.
See also  G. Cz\'edli and E.\,T. Schmidt \cite{CS13} and G. Cz\'edli \cite{gC14}--\cite{gCb}.

A major subchapter of  the theory of planar semimodular lattices 
started with the observation that in the construction of the lattice $K$, 
as in the first paragraph of this section, 
$\SM{3}$ sublattices play a crucial role.
It was natural to raise the question
what can be said about congruence lattices of slim, planar, semimodular (SPS) lattices
(see  [CFL2, Problem~24.1], originally raised in G. Gr\"atzer~\cite{gG14a}). 
In~\cite{gG14a}, I~found the first necessary condition and 
G. Cz\'edli \cite{gC14a} proved that this condition is not sufficient 
(see also my related papers \cite{gG15a} and \cite{gG19}).

A number of papers developed tools to tackle this problem:
the Swing Lemma (G. Gr\"atzer~\cite{gG15}), trajectory coloring (G. Cz\'edli \cite{gC14}),
special diagrams (G. Cz\'edli \cite{gC17}), lamps (G. Cz\'edli \cite{gCa}).
Some of these results require long proofs. 
The proof of the trajectory coloring theorem is just shy of 20 pages, 
while the basic theory of lamps and its application to Theorem~\ref{T:main} is 23 pages.

There are a number of  surveys of this field, 
see the book chapters G.~Cz\'edli and G.~Gr\"atzer~\cite{CG14} 
and  G.~Gr\"atzer~\cite{gG13b} 
in G.~Gr\"atzer and F.~Wehrung, eds.\,~\cite{LTS1}.
My~presentation \cite{gG21a} provides a gentle review
for the background of this topic.

\section{The tools we need}\label{S:Tools}

Most basic concepts and notation not defined in this paper 
are available in Part~I of the book \cite{CFL2}, see 

\verb+https://www.researchgate.net/publication/299594715+\\
\indent {\tt arXiv:2104.06539}

\noindent It is available to the reader. 
We will reference it, for instance, as [CFL2, page 52].
In particular, we use the notation $C \persp D$, $C \perspup D$, 
and $C \perspdn D$ for perspectivity,
up-perspectivity, and down-perspectivity, respectively.
As usual, for planar lattices, a prime interval (or covering interval) is called an \emph{edge}.
For a finite lattice $K$ and a~finite ordered set $R$, 
a \emph{cover-preserving} embedding $\ge \colon R \to K$ is an embedding~$\ge$
mapping edges of $R$ to edges of $K$. 
We define a \emph{cover-preserving} sublattice similarly.
For the lattice $\SfS 7$ of Figure~\ref{F:notation}, 
we need a variant: an  $\SfS 7$ sublattice $\SfS{}$ 
(a sublattice isomorphic to  $\SfS 7$) is a  \emph{peak sublattice}
if the three top edges ($L$, $M$, and $R$ in Figure~\ref{F:notation}) are edges in $K$.

By G. Gr\"atzer and E. Knapp \cite{GKn09}, 
every slim, planar, semimodular lattice $K$ has a congruence-preserving
extension (see [CFL2, page 43])  $\hat K$ to a slim rectangular lattice. 
Any of the properties (i)--(iv) holds for $K$ if{}f it holds for $\hat K$. 
Therefore, in~the rest of this paper, we can assume that $K$ is  a slim rectangular lattice,
simplifying the discussion.

\subsection{Swing Lemma}\label{S:Swing}

For an edge $E$ of an SPS lattice $K$, let $E = [0_E, 1_E]$ 
and define  $\Col{E}$,  the \emph{color of}~$E$, as $\con E$,
the (join-irreducible) congruence generated by collapsing $E$ 
(see [CFL2, Section 3.2]).
We write $\E P$ for $\Ji {\Con K}$, 
the ordered set of join-irreducible congruences of $K$.

As in my paper~\cite{gG15},  for the edges $U, V$ 
of an SPS lattice $K$, we define a binary relation:
$U$~\emph{swings} to $V$, in formula, $U \swing V$, if $1_U = 1_V$, 
the element $1_U = 1_V$ of~$K$ covers at least three elements,
and $0_V$ is neither the left-most 
nor the right-most element covered by $1_U = 1_V$; 
if also $0_U$ is such, then the swing is \emph{interior}, 
otherwise, it is \emph{exterior}, denoted by  $U \inswing V$ and $U \exswing V$, respectively.

\begin{named}{Swing Lemma [G. Gr\"atzer~\cite{gG15}]}
Let $K$ be an SPS lattice and let $U$ and $V$ be edges in $K$. 
Then  $\Col V \leq\Col U$ if{}f there exists an edge $R$ 
such that $U$ is up-perspective to $R$ 
and there exists a sequence of edges and a~sequence of binary relations 
\begin{equation*}\label{E:sequence}
   R = R_0 \grrel_1 R_1 \grrel_2 \dots \grrel_n R_n = V,
\end{equation*}
where each relation $\grrel_i$ is $\perspdn$ \pr{down-perspective} or $\swing$ \pr{swing}.
In~addition, this sequence also satisfies 
\begin{equation*}\label{E:geq}
   1_{R_0} \geq 1_{R_1} \geq \dots \geq 1_{R_n}.
\end{equation*}
\end{named}

The following statements are immediate consequences of the Swing Lemma,
see my papers~\cite{gG15} and \cite{gG14e}.

\begin{corollary}\label{C:equal} 
We use the assumptions of the Swing Lemma.
\begin{enumeratei}
\item The equality $\Col U = \Col V$ holds in $\E P$ if{}f  there exist edges $S$ and $T$ in $K$,
such that 
\begin{equation*}\label{E:xx}
    U \perspup S,\ S \inswing T,\ T \perspdn V.
\end{equation*}
\item Let us further assume that the element $0_U$ is meet-irreducible.
Then the equality $\Col U = \Col V$ holds in $\E P$  
if{}f there exists an edge $T$ such that $U \inswing T \perspdn V$.
\item If the lattice $K$ is rectangular and $U$ is on the upper boundary of $K$, 
then the equality $\Col U = \Col V$ holds in $\E P$ if{}f $U \perspdn V$.
\end{enumeratei}
\end{corollary}

Note that in (i) the edges $S, T, U, V$ need not be distinct, 
so we have as special cases $U = V$, $U \persp V$, $S = T$, and others.

\begin{corollary}\label{C:cov} 
We use the assumptions of the Swing Lemma.
\begin{enumeratei}
\item The covering $\Col V \prec \Col U$ holds in $\E P$  
if{}f there exist edges $R_1, \dots, R_4$ in~$K$, such that 
\begin{equation*}
   U \perspup R_1,\  R_1 \inswing R_2,\  R_2 \perspdn R_3,\ 
      R_3 \exswing R_4,\  R_4 \perspdn V.
\end{equation*}
\item If the element $0_U$ is meet-irreducible,
then the covering $\Col V \prec \Col U$ holds in $\E P$   
if{}f there exist edges $S, T$ in $K$, so that 
\begin{equation*}
   U  \perspdn S \exswing T \perspdn V.
\end{equation*}
\end{enumeratei}
\end{corollary}

\begin{corollary}\label{C:covnew}
Let $K$ be a slim rectangular lattice, let $U$ and $V$ be edges in $K$,
and let  $U$ be in the upper-left boundary of $K$.
\begin{enumeratei}
\item The covering $\Col V \prec \Col U$ holds in $\E P$
if{}f there exist edges $S, T$ in $K$, such that 
\begin{equation}\label{E:seq5}
   U  \perspdn S \exswing T \perspdn V.
\end{equation}
\item Define the element $t = 1_S = 1_T \in K$
and let $S = E_1, E_2, \dots, E_n = W$ enumerate, from left to right, 
all the edges $E$ of $K$  with $1_E = t$. 
Then 
\begin{align}
   \col{S} &\neq  \col{W},\label{E:1}\\
 \col{E_2} = \cdots &=  \col{E_{n-1}} =  \col{T},\label{E:2}\\
  \col{T} &\prec  \col{S}, \col{W}.\label{E:3}
\end{align}
\end{enumeratei}
\end{corollary}

\begin{corollary}\label{C:max}
Let the edge $U$ be on the upper edge of the rectangular lattice $K$. 
Then $\Col U$ is a maximal element of $\E P$.
\end{corollary}

The converse of this statement is stated in Corollary~\ref{C:max1}.

\subsection{$\E C_1$-diagrams}\label{S:diagrams}

In the diagram of a planar lattice $K$,
a \emph{normal edge} (\emph{line}) has a~slope of $45\degree$ or $135\degree$.
If it is the first, we call it a~\emph{normal-up edge} (\emph{line}), 
otherwise, a  \emph{normal-down edge} (\emph{line}).
Any edge of slope strictly between $45\degree$ and $135\degree$ is \emph{steep}.

\begin{definition}\label{D:well}
A diagram of an rectangular lattice $K$ is a
\emph{$\E C_1$-diagram} if the middle edge of any covering $\SfS 7$ is steep
and all other edges are normal.
\end{definition}

This concept was introduced in G.~Cz\'edli~\cite[Definition 5.3(B)] {gC17},
see also G.~Cz\'edli \cite[Definition 2.1]{gCa}
and G. Cz\'edli and G.~Gr\"atzer~\cite[Definition 3.1]{CG21}.
The following is the existence theorem of $\E C_1$-diagrams 
in G. Cz\'edli \cite[Theorem 5.5]{gC17}.

\begin{theorem}\label{T:well}
Every rectangular lattice lattice $K$ has a $\E C_1$-diagram.
\end{theorem}
See the illustrations in this paper for examples of $\E C_1$-diagrams.
For a short and direct proof for the existence of $\E C_1$-diagrams, 
see my paper~\cite{gG21b}.

\emph{In this paper, $K$ denotes a slim rectangular lattice with a fixed $\E C_1$-diagram
and~$\E P$ is the ordered set of join-irreducible congruences of $K$}.

Let $C$ and $D$ be maximal chains in an interval $[a,b]$ of $K$ 
such that $C \ii D = \set{a,b}$.
If there is no element of~$K$ between $C$ and $D$, 
then we call $C \uu D$ a~\emph{cell}. 
A~four-element cell is a \text{\emph{$4$-cell}}. 
Opposite edges of a $4$-cell are called \emph{adjacent}.
Planar semimodular lattices are $4$-cell lattices, 
that is, all of its cells are $4$-cells,
see G.~Gr\"atzer and E. Knapp \cite[Lemmas 4, 5]{GKn07} 
and  [CFL2,~Section 4.1] for more detail.

The following statement illustrates the use of $\E C_1$-diagrams.

\begin{lemma}\label{L:application}
Let $K$ be a slim rectangular lattice $K$ with a fixed $\E C_1$-diagram
and let~$X$ be a normal-up edge of $K$. 
Then $X$ is up-perspective either to an edge in the upper-left boundary of $K$
or to a steep  edge.
\end{lemma}

\begin{proof}
If $X$ is not steep nor it is in the upper-left boundary of $K$, 
then there is a~$4$-cell $C$ whose lower-right edge is $X$.
If the upper-left edge is steep or it is in the upper-left boundary, then we are done. 
Otherwise, we proceed the same way until we reach  a~steep  edge 
or an edge the upper-left boundary.
\end{proof}

\begin{corollary}\label{C:max1}
Let the edge $U$ be on the upper edge of $K$. 
Then $\Col U$ is a maximal element of $\E P$.
Conversely, if $u$ is a maximal element of $\E P$,
then there is an edge $U$ on the upper edge of $K$ so that $\Col U = u$.
\end{corollary}

\subsection{Trajectories}\label{S:Trajectories}
G. Cz\'edli and E.\,T. Schmidt \cite{CS11} introduced a \emph{trajectory} in $K$ 
as a maximal sequence of consecutive edges, see also [CFL2, Section~4.1]. 
The \emph{top edge}~$T$ of a trajectory 
is either in the upper boundary of $K$ or it is steep by Lemma~\ref{L:application}. 
For such an edge~$T$, we denote by $\traj T$ the trajectory with top edge~$T$.

By G.~Gr\"atzer and E. Knapp \cite[Lemma 8]{GKn07}, 
an element $a$ in an SPS lattice~$K$ has at most two covers.
Therefore, a trajectory has at most one top edge and at most one steep edge.
So we conclude the following statement.

\begin{lemma}\label{L:disj}
Let $K$ be a slim rectangular lattice $K$ with a fixed $\E C_1$-diagram.
Let $X$ and $Y$ be distinct steep edges of $K$. 
Then $\traj X$ and $\traj Y$ are disjoint.
\end{lemma}

\section{The Partition Property}\label{S:partition}

First, we verify the Partition Property for the slim rectangular lattice $K$ 
and with a fixed $\E C_1$-diagram.
We start with a lemma.

\begin{lemma}\label{L:disjoint}
Let $X $ and $Y$ be distinct edges on the upper-left boundary of $K$. 
Then there is no edge $Z$ of $K$ such that $\col Z \prec \col X, \col Y$.
\end{lemma}

\begin{proof}
By way of contradiction, let $Z$ be an edge such that $\col Z \prec \col X, \col Y$.
Since $X$ and $Y$ are on the upper-left boundary, 
Corollary~\ref{C:covnew}(i) applies. 
Therefore, there exist normal-up edges $S_X, S_Y$ and steep edges $T_X, T_Y$ such that 
\[
   X \perspdn S_X \exswing T_X,\q  Y \perspdn S_Y \exswing T_Y,\q  
   Z \in \traj {T_X} \ii \traj {T_Y}.
\] 
By Lemma~\ref{L:disj}, the third formula implies that $T_X = T_Y$ 
and xo $X = Y$, contrary to the assumption.
\end{proof}

By Corollary~\ref{C:max1}, the set of maximal elements of $\E P$ is the same 
as the set of colors of edges in the upper boundaries.
We can partition the set of edges in the upper boundaries
into the set of edges~$\E L$ in the upper-left boundary
and the set of edges~$\E R$ in the upper-right boundary.
If $X $ and $Y$ are distinct edges in $\E L$, 
then there is no edge $Z$ of $K$ such that $\col Z \prec \col X, \col Y$
by Lemma~\ref{L:disjoint}. 
By symmetry, this verifies the Partition Property.

\section{The Maximal Cover Property}\label{S:Maximal}

Next, we verify the Maximal Cover Property for the slim rectangular lattice~$K$ 
and with a fixed $\E C_1$-diagram.

Let $x \in \E P$ be covered by a maximal element $u$ of $\E P$ in $K$. 
By Corollary~\ref{C:max1}, we can choose an edge $U$ of color $u$ on the upper boundary of $K$,
by symmetry, on the upper-left boundary of $K$. 
By Corollary~\ref{C:covnew}(ii), we can choose the edges $S, T$ in $K$ 
so that  $U  \perspdn S \exswing T$, $\col S = u$, and $\col T = x$. 
By Corollary~\ref{C:covnew}(ii), specifically, by equations \eqref{E:1} and \eqref{E:3},
we have $x \prec  u, \col{W}$ and $u \neq \col{W}$, verifying the Maximal Cover Property.

\section{The No Child Property}\label{S:Child}

In this section, we verify the No Child Property for the slim rectangular lattice~$K$ 
and with a fixed $\E C_1$-diagram.

Let $x,y,z,u \in \E P$ with $x \neq y \in \E P$,
let $u$ be a maximal element of $\E P$,
and let $x, y \prec u$ in $\E P$. 
By way of contradiction, 
let us assume that there is an element $z \in \E P$ 
such that $z \prec x,y$ in $\E P$.

By Corollary~\ref{C:max1}, the element~$u$ colors an edge~$U$ on the upper boundary of~$K$,
say, in the upper-left boundary. 
By Corollary~\ref{C:cov}(i),  for  $z \prec x \in \E P$, 
we get a peak sublattice $\seven$ 
in which the middle edge $Z$ is colored by $z$ and upper-left edge $X$
is colored by $x$, or symmetrically. 
The upper-right edge $Y$ must have color~$y$.

Now we apply Corollary~\ref{C:covnew}(ii) 
to the edge $U$ and middle edge $Z$ of the peak sublattice $\seven$,
obtaining that $U \perspdn Y \swing Z$, in particular, $U \perspdn Y$.
This is a contradiction, since $U$ is normal-up and $Y$ is normal-down.

\section{The Four-Crown Two-pendant Property}\label{S:Crown}

Finally, we verify the  Four-Crown Two-pendant Property for the slim rectangular lattice $K$ 
and with a fixed $\E C_1$-diagram.

By way of contradiction, assume that the ordered set $\E R$ of Figure~\ref{F:notation}
is a cover-preserving ordered subset of $\E P$,
where $a,b,c,d$ are maximal elements of $\E P$. 
By~Corollary~\ref{C:max1}, there are edges $A,B,C,D$ on the upper boundary of $K$, 
so that  $\col A = a$, $\col B = b$, $\col C=c$, $\col D = d$.
By left-right symmetry, we can assume that the edge $A$ 
is on the upper-left boundary of $K$. Since $p \prec a, b$ in  $\E P$,
it follows from Lemma~\ref{L:disjoint} that the edge $B$
is on the upper-right boundary of $K$, and so is $D$.
Similarly, $C$ is on the upper-left boundary of $K$. 

There are four cases, (i) $C$ is below $A$ and $B$ is below $D$;
(ii)~$C$~is below $A$ and $D$ is below $B$; and so on. 
The first two are illustrated in Figure~\ref{F:CABDx}.

\begin{figure}[htb]
\centerline{\includegraphics[scale=1.2]{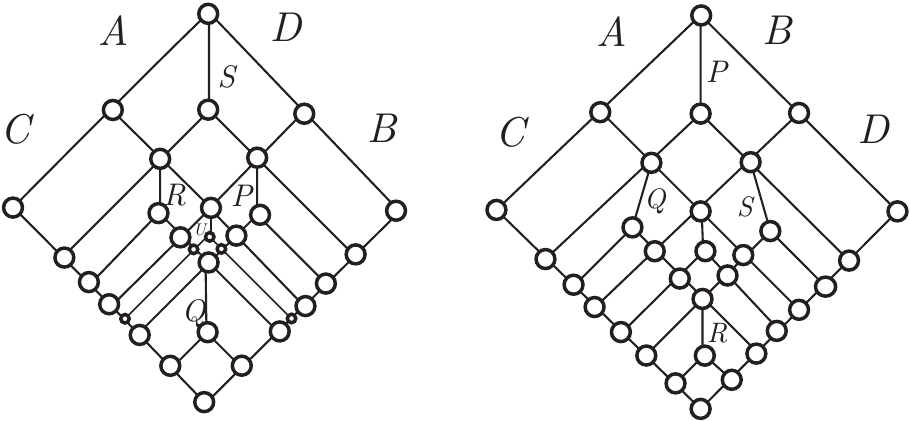}}
\caption{Illustrating the proof of The Four-Crown Two-pendant Property}
\label{F:CABDx}
\end{figure}

We consider the first case.
By Corollary~\ref{C:cov}(ii),  
there is a peak sublattice $\SfS 7$ with middle edge $P$ 
(as in the first diagram of Figure~\ref{F:CABDx})
so that $A$ and $B$ are down-perspective 
to the upper-left edge and the upper-right edge of this peak sublattice, respectively. 
We define, similarly, the edge $Q$ for $C$ and $B$,
the edge $S$ for $A$ and~$D$, the edge $R$ for $C$ and $D$,
and the edge $U$ for $R$ and $P$.

The ordered set $\E R$ is a cover-preserving subset of $\E P$,
so we get, similarly, the peak sublattice~$\SfS 7$ with middle edge $U$.
Finally, $v \prec q, s$ in $\E R$, 
therefore, there is a peak sublattice~$\SfS 7$ with middle edge $V$
with upper-left edge $V_l$ and the upper-right edge~$V_r$
so that  $S \perspdn V_l$ and $S \perspdn V_r$, or symmetrically.

This concludes the proof of the Four-Crown Two-pendant Property
and of Cz\'edli's Theorem.

Of course, the diagrams in Figure~\ref{F:CABDx} are only illustrations.
The grid could be much larger, the edges $A, C$ and $B, D$ may not be adjacent, 
and there maybe lots of other elements in $K$. 
However, our argument does not utilize the special circumstances in the diagrams.

The second case is similar, except that we get the edge $V$ 
and cannot get the edge $U$.
The third and fourth cases follow the same way.

\appendix

\section{Two more illustrations for Section~\ref{S:Crown}}\label{S:appendix}

\begin{figure}[htb]
\centerline{\includegraphics[scale=1.2]{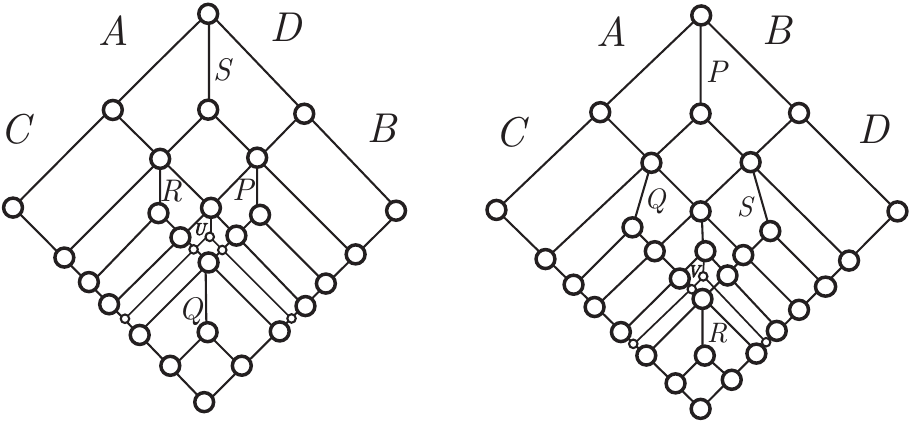}}
\caption{Two more illustrations for Section~\ref{S:Crown}}
\label{F:CABDx2}
\end{figure}

\end{document}